\documentclass[10pt]{amsart}
\usepackage{amssymb}
\usepackage{amscd}
\usepackage[all]{xy}

\newcounter{TmpEnumi}

\numberwithin{equation}{section}

\def\today{\number\day\space\ifcase\month\or   January\or February\or
   March\or April\or May\or June\or   July\or August\or September\or
   October\or November\or December\fi\   \number\year}

\theoremstyle{definition}
\newtheorem{thm}{Theorem}[section]
\newtheorem{lem}[thm]{Lemma}
\newtheorem{prp}[thm]{Proposition}
\newtheorem{dfn}[thm]{Definition}
\newtheorem{cor}[thm]{Corollary}

\newtheorem{rmk}[thm]{Remark}
\newtheorem{ntn}[thm]{Notation}
\newtheorem{exa}[thm]{Example}

\newcommand{\beq}{\begin{equation}}
\newcommand{\eeq}{\end{equation}}
\newcommand{\beqr}{\begin{eqnarray*}}
\newcommand{\eeqr}{\end{eqnarray*}}
\newcommand{\bal}{\begin{align*}}
\newcommand{\eal}{\end{align*}}
\newcommand{\bei}{\begin{itemize}}
\newcommand{\eei}{\end{itemize}}

\newcommand{\af}{\alpha}
\newcommand{\bt}{\beta}

\newcommand{\ld}{\lambda}

\newcommand{\ph}{\varphi}

\newcommand{\Ld}{\Lambda}

\newcommand{\Z}{{\mathbb{Z}}}

\newcommand{\C}{{\mathbb{C}}}
\newcommand{\N}{{\mathbb{Z}}_{> 0}}

\newcommand{\id}{{\operatorname{id}}}

\newcommand{\card}{{\operatorname{card}}}
\newcommand{\Aut}{{\operatorname{Aut}}}
\newcommand{\Ad}{{\operatorname{Ad}}}

\newcommand{\cC}{{\mathcal{C}}}

\newcommand{\cP}{{\mathcal{P}}}

\newcommand{\tfae}{the following are equivalent}
\newcommand{\ifo}{if and only if}

\newcommand{\ca}{C*-algebra}

\newcommand{\fd}{finite dimensional}

\newcommand{\pj}{projection}

\newcommand{\cp}{crossed product}

\newcommand{\ssa}{separable unital strongly selfabsorbing}

\renewcommand{\S}{\subset}

\newcommand{\I}{\infty}
\newcommand{\E}{\varnothing}

\newcommand{\LI}[2]{{\operatorname{I}}_{#1} (#2)}

\title[Crossed products and fixed point algebras]{Relating
  properties of crossed products to those of
  fixed point algebras}

\author{Cornel Pasnicu and N.~Christopher Phillips}

\date{30~August 2017}

\address{Department of Mathematics,
      The University of Texas at San Antonio,
      San Antonio TX 78249, USA.}
\address{Department of Mathematics, University  of Oregon,
       Eugene OR 97403-1222, USA.}

\thanks{This material is based upon work supported by the
  US National Science Foundation under
  Grant DMS-1501144.}

\begin{document}

\begin{abstract}
For a number of properties of C*-algebras,
including real rank zero, stable rank one,
pure infiniteness,
residual hereditary infiniteness,
the combination of pure infiniteness and the ideal property,
the property of being an AT~algebra with real rank zero,
and $D$-stability for a \ssa{} \ca~$D$,
we prove the following.
Let $A$ be a separable \ca,
let $G$ be a second countable compact abelian group,
and let $\af \colon G \to \Aut (A)$
be any action of $G$ on~$A$.
Then the fixed point algebra $A^{\af}$ has the given property
if and only if
the crossed product $C^* (G, A, \af)$ has the same property.
\end{abstract}

\maketitle

\indent
Let $(P)$ be a property that a \ca{} may of may not have,
such as real rank zero, satisfying the Universal Coefficient Theorem,
the weak ideal property,
or being $D$-stable for a \ssa{} \ca~$D$.
Let $G$ be a compact abelian group,
and let $\af \colon G \to \Aut (A)$
be an action of $G$ on a \ca~$A$.
(No freeness conditions are assumed on the action.)
Suppose the fixed point algebra $A^{\af}$ has the property $(P)$.
Does it follow that the \cp{} $C^* (G, A, \af)$ has $(P)$?
In this paper,
assuming separability conditions,
we give positive answers for a number of choices of $(P)$,
including real rank zero, stable rank one,
pure infiniteness,
residual hereditary infiniteness,
the combination of pure infiniteness and the ideal property,
the property of being an AT~algebra with real rank zero,
and $D$-stability for a \ssa{} \ca~$D$.
See Proposition~\ref{P_7809_IsMaxx},
Remark~\ref{R_7809_ResidHer},
and Lemma~\ref{L_7812_UCT}
for a list of the properties we consider.
For type~I \ca{s},
this result
(without $G$ being abelian and without separability)
is in Theorem~3.2 of~\cite{GtLz3}.

This situation is easier,
and holds for many more properties,
than showing that if $A$ has $(P)$
then $C^* (G, A, \af)$ has $(P)$.

The reverse direction,
showing that if $C^* (G, A, \af)$ has $(P)$
then $A^{\af}$ has $(P)$,
is easy for the properties we consider:
they pass to corners,
and $A^{\af}$ is isomorphic to a corner in $C^* (G, A, \af)$,
as in~\cite{Rs}.
The question above will usually be easy when
$A^{\af}$ is isomorphic to a full corner in $C^* (G, A, \af)$,
which is expected in the presence of mild freeness conditions.
In general $A^{\af}$ need not be full,
but in some sense the parts of $C^* (G, A, \af)$
not in the ideal generated by the image of $A^{\af}$
look like those that are in this ideal.
One may consider the following examples,
in which $G = \Z / 2 \Z$ with nontrivial element $g_0 \in G$.
\begin{enumerate}
\item\label{E_7811_Rep_Triv}
$A = \C$ and $\af \colon G \to \Aut (A)$
is determined by $\af_{g_0} = \id_{\C}$.
Then $A^{\af} = \C$ and $C^* (G, A, \af) \cong \C \oplus \C$.
The isomorphism of~\cite{Rs} identifies $A^{\af}$ with one of the
summands in $C^* (G, A, \af)$.
It is not full.
\item\label{E_7811_Rep_Rokhlin}
$B$ is the $2^{\infty}$~UHF algebra and $\bt \colon G \to \Aut (B)$
is determined by
$\bt_{g_0} = \bigotimes_{n = 1}^{\I}
 \Ad \left( \left( \begin{smallmatrix}
  1     &  0        \\
  0     &  -1
\end{smallmatrix} \right) \right)$.
One can check that $B^{\bt} \cong C^* (G, B, \bt) \cong B$,
and that the isomorphism of~\cite{Rs} identifies $B^{\bt}$
with a (full) corner in $C^* (G, B, \bt)$
given by a \pj{} with trace~$\frac{1}{2}$.
\end{enumerate}

The properties we consider share the following characteristics,
which allow them to be treated all at once.
They pass to ideals,
they pass to stably isomorphic \ca{s},
and every \ca{} has a largest ideal with the given property.
The key technical point is that in a crossed product by
an abelian group,
this largest ideal must be invariant under the dual action.

In Section~\ref{Sec_Def},
we give the definitions
and some general theory.
In Section~\ref{Sec_Max},
we verify that a number of properties of \ca{s}
fit in our framework.
Section~\ref{Sec_Consq}
contains the main result,
relating the fixed point algebra to the crossed product.
It is interesting
primarily for properties which are not preserved
by crossed products.
For those that are,
we relate the largest ideal with a given property in an algebra
to the largest ideal in the crossed product
with the same property.

As usual,
$K$ denotes the algebra of compact operators
on a separable in\fd{} Hilbert space.
All ideals in \ca{s} are assumed to be two sided and closed.

\section{Largest ideals}\label{Sec_Def}

\indent
In this section,
we define various conditions which may hold for a
property of \ca{s}.
Most are known or obvious,
and are included for completeness.
The new one is admitting largest ideals
(Definition \ref{D_7809_Props}(\ref{D_7809_Props_Max})).
We then give several results showing that other conditions
imply that a property admits largest ideals.
We need several of these conditions later.
At the end of the section,
we show that, under extra hypotheses,
largest ideals with a given property behave as expected.

\begin{dfn}\label{D_7809_Props}
Let $(P)$ be a property that a \ca{} may of may not have
(such as real rank zero, satisfying the Universal Coefficient Theorem,
or the weak ideal property).
\begin{enumerate}
\item\label{D_7809_Props_Iso}
We say that $(P)$ is {\emph{isomorphism invariant}}
if whenever $A$ and $B$ are \ca{s},
$A$ has the property~$(P)$,
and $B \cong A$, then $B$ has the property~$(P)$.
\item\label{D_7809_Props_Stable}
We say that $(P)$ is {\emph{(separably) stable}}
if whenever $A$ and $B$ are (separable) \ca{s},
$A$ has the property~$(P)$,
and $K \otimes B \cong K \otimes A$, then $B$ has the property~$(P)$.
\item\label{D_7809_Props_Ideals}
We say that $(P)$ {\emph{(separably) passes to ideals}}
if whenever $A$ is a (separable) \ca{}
with~$(P)$
and $I \subseteq A$ is an ideal,
then $I$ has the property~$(P)$.
\item\label{D_7809_Props_Max}
We say that $(P)$ {\emph{(separably) admits largest ideals}}
if for every (separable) \ca~$A$
there is a largest ideal in~$A$
which has the property~$(P)$.
\end{enumerate}
\end{dfn}

More explicitly,
a property $(P)$ admits largest ideals
if for every \ca~$A$
there is an ideal $I \subseteq A$
which has~$(P)$,
and such that whenever $J \subseteq A$ is an ideal with~$(P)$,
then $J \subseteq I$.

\begin{ntn}\label{N_7809_MaxI}
Let $(P)$ be a property which (separably) admits largest ideals.
For a (separable) \ca~$A$,
we denote by $\LI{P}{A}$
the largest ideal in~$A$ which has~$(P)$.
\end{ntn}

We now give three progressively stronger sets of conditions
which imply that a property admits largest ideals.
For some later results (not the main result),
we will need to assume one of these stronger sets of conditions.

\begin{lem}\label{L_7809_CrtForMax}
Let $(P)$ be a property of \ca{s}.
Assume that:
\begin{enumerate}
\item\label{L_7809_CrtForMax_Z}
The zero \ca{} has~$(P)$.
\item\label{L_7809_CrtForMax_Sum}
Whenever $A$ is a \ca{} and $I, J \subseteq A$ are ideals with~$(P)$,
then there is an ideal $L \subseteq A$ which has~$(P)$
and such that $I, J \subseteq L$.
\item\label{L_7809_CrtForMax_Dir}
Whenever $A$ is a \ca,
$\Ld$ is a directed set,
and $(I_{\ld})_{\ld \in \Ld}$ is a family of ideals in~$A$ with~$(P)$
such that $I_{\ld} \subseteq I_{\mu}$
for all $\ld, \mu \in \Ld$
with $\ld \leq \mu$,
then there is an ideal $L \subseteq A$ which has~$(P)$
and such that $I_{\ld} \subseteq L$ for all $\ld \in \Ld$.
\end{enumerate}
Then $(P)$ admits largest ideals.
\end{lem}

\begin{proof}
Let $A$ be a \ca{}
and let $\cP$ be the collection of all ideals in $A$
which have~$(P)$,
ordered by inclusion.
Condition~(\ref{L_7809_CrtForMax_Z}) implies that
$\cP \neq \E$.
Condition~(\ref{L_7809_CrtForMax_Dir}) implies that
every totally ordered subset of $\cP$
has an upper bound.
So Zorn's Lemma provides a maximal element $M$ of~$\cP$.
Now let $I \in \cP$.
Then condition~(\ref{L_7809_CrtForMax_Sum})
provides an ideal $L \in \cP$
such that $M \subseteq L$ and $I \subseteq L$.
By maximality, $L = M$, so $I \subseteq M$.
Thus $M$ is the largest element of~$\cP$.
\end{proof}

\begin{cor}\label{C_7809_StCtForMaxx}
Let $(P)$ be a property of \ca{s}.
Assume that:
\begin{enumerate}
\item\label{C_7809_StCtForMaxx_Z}
The zero \ca{} has~$(P)$.
\item\label{C_7809_StCtForMaxx_Sum}
Whenever $A$ is a \ca{} and $I, J \subseteq A$ are ideals with~$(P)$,
then $I + J$ has~$(P)$.
\item\label{C_7809_StCtForMaxx_Dir}
Whenever $A$ is a \ca,
$\Ld$ is a directed set,
and $(I_{\ld})_{\ld \in \Ld}$ is a family of ideals in~$A$ with~$(P)$
such that $I_{\ld} \subseteq I_{\mu}$
for all $\ld, \mu \in \Ld$
with $\ld \leq \mu$,
then ${\overline{\bigcup_{\ld \in \Ld} I_{\ld} }}$ has~$(P)$.
\end{enumerate}
Then $(P)$ admits largest ideals.
\end{cor}

\begin{proof}
The conditions imply those of Lemma~\ref{L_7809_CrtForMax}.
\end{proof}

\begin{cor}\label{C_7809_ExtCtForMaxx}
Let $(P)$ be a property of \ca{s}.
Assume that:
\begin{enumerate}
\item\label{C_7809_ExtStCtForMaxx_Z}
The zero \ca{} has~$(P)$.
\item\label{C_7809_ExtStCtForMaxx_Quot}
Quotients of algebras with~$(P)$ have~$(P)$.
\item\label{C_7809_ExtStCtForMaxx_Ext}
Extensions of algebras with~$(P)$ have~$(P)$.
\item\label{C_7809_ExtStCtForMaxx_Dir}
Whenever $A$ is a \ca,
$\Ld$ is a directed set,
and $(I_{\ld})_{\ld \in \Ld}$ is a family of ideals in~$A$ with~$(P)$
such that $I_{\ld} \subseteq I_{\mu}$
for all $\ld, \mu \in \Ld$
with $\ld \leq \mu$,
then ${\overline{\bigcup_{\ld \in \Ld} I_{\ld} }}$ has~$(P)$.
\end{enumerate}
Then $(P)$ admits largest ideals.
\end{cor}

\begin{proof}
It suffices to verify condition~(\ref{C_7809_StCtForMaxx_Sum})
of Corollary~\ref{C_7809_StCtForMaxx}.
So let $I, J \subseteq A$ be ideals with $(P)$.
Using the
short exact sequence
\[
0 \longrightarrow I
\longrightarrow I + J
\longrightarrow J / (I \cap J)
\longrightarrow 0
\]
and hypotheses (\ref{C_7809_ExtStCtForMaxx_Quot})
and (\ref{C_7809_ExtStCtForMaxx_Ext}),
it follows that $I + J$ has $(P)$.
\end{proof}

We need the separable versions
of Corollary~\ref{C_7809_StCtForMaxx}
and Corollary~\ref{C_7809_ExtCtForMaxx}.

\begin{lem}\label{L_7813_ExtCtForDSepMaxx}
Let $(P)$ be a property of \ca{s}.
Assume that:
\begin{enumerate}
\item\label{L_7813_ExtCtForDSepMaxx_Z}
The zero \ca{} has~$(P)$.
\item\label{L_7813_ExtCtForDSepMaxx_Sum}
Whenever $A$ is a separable \ca{}
and $I, J \subseteq A$ are ideals with~$(P)$,
then $I + J$ has~$(P)$.
\item\label{L_7813_ExtCtForDSepMaxx_Dir}
Whenever $A$ is a separable \ca{}
and $I_0 \subseteq I_1 \subseteq \cdots \subseteq A$
is a sequence of ideals with~$(P)$,
then ${\overline{\bigcup_{n = 0}^{\I} I_{n} }}$ has~$(P)$.
\end{enumerate}
Then $(P)$ separably admits largest ideals.
\end{lem}

\begin{proof}
Let $\cP$ be the collection of all ideals in $A$
which have~$(P)$.
Choose a countable dense subset $\{ a_n \colon n \in \N \}$ of
$\bigcup_{I \in \cP} I$.
Inductively construct ideals $I_n \in \cP$ as follows.
Define $I_0 = \{ 0 \}$.
Given $I_n \in \cP$,
choose an ideal $J \in \cP$
such that $a_{n + 1} \in J$.
Then set $I_{n + 1} = I_n + J$,
which is in $\cP$ by hypothesis.

We have $I_0 \subseteq I_1 \subseteq \cdots \subseteq A$,
so the ideal $I = {\overline{\bigcup_{n = 0}^{\I} I_{n} }}$
is in $\cP$ by hypothesis.
Also
$L \in \cP$ implies
$L \subseteq {\overline{ \{ a_n \colon n \in \N \} }} \subseteq I$.
So $I$ is the largest ideal in $A$ which has~$(P)$.
\end{proof}

\begin{cor}\label{C_7826_Sep}
Let $(P)$ be a property of \ca{s}.
Assume that:
\begin{enumerate}
\item\label{C_7826_Sep_Z}
The zero \ca{} has~$(P)$.
\item\label{C_7826_Sep_Quot}
Quotients of separable algebras with~$(P)$ have~$(P)$.
\item\label{C_7826_Sep_Ext}
Extensions of separable algebras with~$(P)$ have~$(P)$.
\item\label{C_7826_Sep_Dir}
Whenever $A$ is a separable \ca{}
and $I_0 \subseteq I_1 \subseteq \cdots \subseteq A$
is a sequence of ideals with~$(P)$,
then ${\overline{\bigcup_{n = 0}^{\I} I_{n} }}$ has~$(P)$.
\end{enumerate}
Then $(P)$ separably admits largest ideals.
\end{cor}

\begin{proof}
The proof is the same as that of
Corollary~\ref{C_7809_ExtCtForMaxx}.
\end{proof}

The corollary to the following proposition
is the expected behavior of a
property which admits largest ideals.
However,
we don't know how to prove it without the additional hypotheses
in the statement,
and we suppose that in general it fails.

\begin{prp}\label{P_7809_Quot}
Let $(P)$ be a property of \ca{s}
which satisfies the conditions of Corollary~\ref{C_7809_ExtCtForMaxx}.
Let $A$ be a \ca{}
and let $I \subseteq A$ be an ideal.
Then $\LI{P}{A / \LI{P}{I}} = \LI{P}{A} / \LI{P}{I}$.
If $(P)$ instead
satisfies the conditions of Corollary~\ref{C_7826_Sep},
the same conclusion holds provided that $A$ is separable.
\end{prp}

\begin{proof}
We prove only the first part.
The proof of the last statement is the same,
except that we require separability throughout.

The algebra $\LI{P}{A} / \LI{P}{I}$ is an ideal in $A / \LI{P}{I}$.
It has~$(P)$ since $\LI{P}{A}$ does and $(P)$ passes to quotients.

Now let $J \subseteq A / \LI{P}{I}$ be an ideal with~$(P)$.
We need to show that $J \subseteq \LI{P}{A} / \LI{P}{I}$.
Let $L \subseteq A$ be the inverse image of $J$ under the quotient map.
In the short exact sequence
\[
0 \longrightarrow \LI{P}{I}
\longrightarrow L
\longrightarrow J
\longrightarrow 0,
\]
$\LI{P}{I}$ and $J$ have~$(P)$,
so $L$ has~$(P)$ by hypothesis.
Therefore $L \subseteq \LI{P}{A}$,
whence $J \subseteq \LI{P}{A} / \LI{P}{I}$.
\end{proof}

\begin{cor}\label{C_7810_LILI}
Let $(P)$ be a property of \ca{s}
which satisfies the conditions of Corollary~\ref{C_7809_ExtCtForMaxx}.
Let $A$ be a \ca.
Then $A / \LI{P}{A}$ has no nonzero ideals with~$(P)$.
If $(P)$ instead
satisfies the conditions of Corollary~\ref{C_7826_Sep},
the same conclusion holds provided that $A$ is separable.
\end{cor}

\begin{proof}
In Proposition~\ref{P_7809_Quot}, take $I = A$.
\end{proof}

\section{Properties admitting largest ideals}\label{Sec_Max}

Our main result on crossed products
(Theorem~\ref{T_7809_FixedToCP})
requires
a property~$(P)$ that separably passes to ideals,
is separably stable, and separably admits largest ideals.
In this section,
we prove that various interesting properties
satisfy these conditions.
In Proposition~\ref{P_7809_IsMaxx},
we give a list for which separability is not needed,
but we postpone the proof until after some lemmas.
The properties for which we need separability
are in Lemma~\ref{L_7812_UCT}.

It is tempting to include
``being an AI~algebra with real rank zero''
as one of the properties,
by analogy with Lemma \ref{L_7812_UCT}(\ref{L_7812_UCT_ATZ}),
but this property is equivalent to being AF
by Corollary 3.2.17(i) of~\cite{Rrd}.

For Proposition \ref{P_7809_IsMaxx}(\ref{P_7809_IsMaxx_PI}),
recall that pure infiniteness for nonsimple \ca{s}
is defined in Definition 4.1 of~\cite{KR},
and for
Proposition \ref{P_7809_IsMaxx}(\ref{P_7809_IsMaxx_StPI}),
recall that strong pure infiniteness
is defined in Definition 5.1 of~\cite{KrRd2}.
In connection with
Proposition \ref{P_7809_IsMaxx}(\ref{P_7809_IsMaxx_ResidHer}),
we recall from Definition 5.1 of~\cite{PsnPh2}
that a class $\cC$ of \ca{s} is
{\emph{upwards directed}} if
whenever $A$ is a C*-algebra which contains a subalgebra isomorphic to
an algebra in~$\cC$,
then $A \in \cC$,
and from Definition 5.2 of~\cite{PsnPh2} that a \ca~$A$
is {\emph{residually hereditarily in}} such a class~$\cC$
if for every ideal $I \subseteq A$
and every nonzero hereditary subalgebra $B \subseteq A / I$,
we have $B \in \cC$.
See Remark~\ref{R_7809_ResidHer}
for some properties which have this form.

Nuclearity and type~I,
and to a lesser extent AF,
are included partly to relate our methods
to known results.
In particular,
we prove nothing new about nuclear or type~I \ca{s}.
Theorem~3.2 of~\cite{GtLz3} gives a much stronger statement
than our Theorem~\ref{T_7809_FixedToCP}
for type~I \ca{s}:
$G$ is only required to be compact
(not necessarily abelian or second countable),
and $A$ need not be separable.

\begin{prp}\label{P_7809_IsMaxx}
Each of the following properties
is stable, passes to ideals, and admits largest ideals:
\begin{enumerate}
\item\label{P_7809_IsMaxx_T1}
Type~I.
\item\label{P_7809_IsMaxx_Nuc}
Nuclearity.
\item\label{P_7809_IsMaxx_tsr1}
Stable rank one.
\item\label{P_7809_IsMaxx_RRZ}
Real rank zero.
\item\label{P_7809_IsMaxx_PI}
Pure infiniteness for nonsimple \ca{s}.
\item\label{P_7809_IsMaxx_StPI}
Strong pure infiniteness for nonsimple \ca{s},
provided that the zero \ca{} is included.
\item\label{P_7809_IsMaxx_Tor}
All ideals have torsion K-theory.
\item\label{P_7809_IsMaxx_Van}
All ideals have zero K-theory.
\item\label{P_7809_IsMaxx_RRZTor}
Real rank zero and all ideals have torsion $K_0$-groups.
\item\label{P_7809_IsMaxx_RRZTriv}
Real rank zero and all ideals have trivial $K_0$-groups.
\item\label{P_7809_IsMaxx_Sr1Tor}
Stable rank one and all ideals have torsion $K_1$-groups.
\item\label{P_7809_IsMaxx_Sr1Triv}
Stable rank one and all ideals have trivial $K_1$-groups.
\item\label{P_7809_IsMaxx_TDZ}
Topological dimension zero.
\item\label{P_7809_IsMaxx_ResidHer}
Being residually hereditarily in
a fixed upwards directed class of \ca{s}.
\end{enumerate}
\end{prp}

We list some properties of the form in~(\ref{P_7809_IsMaxx_ResidHer}).
The proofs are in~\cite{PsnPh2}.
A convenient summary,
with explicit references,
is given at the end of Section~1 of~\cite{PsnPh3}.

\begin{rmk}\label{R_7809_ResidHer}
The following properties are covered
by Proposition \ref{P_7809_IsMaxx}(\ref{P_7809_IsMaxx_ResidHer}):
\begin{enumerate}
\item\label{R_7809_ResidHer_WIP}
The weak ideal property:
in the stabilization of the algebra,
every nonzero quotient of one ideal by another contains a nonzero
projection.
\item\label{R_7809_ResidHer_IPPI}
The combination of pure infiniteness for nonsimple \ca{s}
and the ideal property.
\item\label{R_7809_ResidHer_HerInf}
Residual hereditary infiniteness:
every nonzero hereditary subalgebra in every quotient
contains an infinite element
in the sense of Definition~3.2 of~\cite{KR}.
\item\label{R_7809_ResidHer_HerPropInf}
Residual hereditary proper infiniteness:
every nonzero hereditary subalgebra in every quotient
contains a properly infinite element
in the sense of Definition~3.2 of~\cite{KR}.
\item\label{R_7809_ResidHer_RSP}
Residual~(SP):
every nonzero hereditary subalgebra in every quotient
contains a nonzero \pj.
\end{enumerate}
\end{rmk}

\begin{exa}\label{E_7809_NotHered}
The ideal property admits largest ideals
(Proposition 6.2 of~\cite{CP})
but is
not separably stable
(see Example~2.8 of~\cite{PsnPh1}).
\end{exa}

\begin{lem}\label{L_7810_SatCond}
Each of the following properties satisfies
the conditions of Corollary~\ref{C_7809_StCtForMaxx}:
\begin{enumerate}
\item\label{L_7810_SatCond_tsr1}
Stable rank one.
\item\label{L_7810_SatCond_RRZ}
Real rank zero.
\item\label{L_7810_SatCond_Tor}
All ideals have torsion K-theory.
\item\label{L_7810_SatCond_Van}
All ideals have zero K-theory.
\item\label{L_7810_SatCond_RRZTor}
Real rank zero and all ideals have torsion $K_0$-groups.
\item\label{L_7810_SatCond_RRZTriv}
Real rank zero and all ideals have trivial $K_0$-groups.
\item\label{L_7810_SatCond_Sr1Tor}
Stable rank one and all ideals have torsion $K_1$-groups.
\item\label{L_7810_SatCond_Sr1Triv}
Stable rank one and all ideals have trivial $K_1$-groups.
\end{enumerate}
\end{lem}

\begin{proof}
The zero algebra certainly has all of these properties.

We claim that all of the properties are preserved by
arbitrary direct limits.
This is well known
(and easy to prove)
for stable rank one and real rank zero.
For properties (\ref{L_7810_SatCond_Tor}),
(\ref{L_7810_SatCond_Van}),
(\ref{L_7810_SatCond_RRZTor}),
(\ref{L_7810_SatCond_RRZTriv}),
(\ref{L_7810_SatCond_Sr1Tor}),
and~(\ref{L_7810_SatCond_Sr1Triv}),
we use the following facts:
ideals in direct limit algebras are direct limits of ideals
in the algebras in the system;
K-theory commutes with direct limits;
and direct limits of torsion groups or trivial groups
are torsion groups or trivial groups.

It remains to prove that
if $A$ is a \ca{} and $I, J \subseteq A$ are ideals with
one of these properties,
then $I + J$ has the same property.
For stable rank one this is
Proposition 2.14(ii) of~\cite{BP2},
and for real rank zero it is
Proposition 2.14(iii) of~\cite{BP2}.

For the other properties, consider the
short exact sequence
\[
0 \longrightarrow I
\longrightarrow I + J
\longrightarrow J / (I \cap J)
\longrightarrow 0.
\]
For (\ref{L_7810_SatCond_RRZTor}),
(\ref{L_7810_SatCond_RRZTriv}),
(\ref{L_7810_SatCond_Sr1Tor}),
and~(\ref{L_7810_SatCond_Sr1Triv}),
we have to consider an arbitrary ideal $L \subseteq I + J$.
One uses an approximate identity for~$L$ to check that
$L = I \cap L + J \cap L$.
The ideals $I \cap L$ and $J \cap L$
have real rank zero or stable rank one as appropriate,
because these properties pass to ideals,
and they have torsion or trivial $K_0$ or $K_1$~groups
as appropriate by hypothesis.
In the real rank zero cases,
we now appeal to Lemma~2.4 of~\cite{BP2}.
Since $I \cap L$ has real rank zero,
it says that
$K_0 (I \cap L) \oplus K_0 (J \cap L) \to K_0 (L)$
is surjective.
So if $K_0 (I \cap L)$ and $K_0 (J \cap L)$ are both torsion,
or both trivial,
so is $K_0 (L)$.
In the stable rank one cases,
we appeal instead to Lemma~5.2 of~\cite{BnPd1},
which similarly gives surjectivity of
$K_1 (I \cap L) \oplus K_1 (J \cap L) \to K_1 (L)$,
and argue in the same way.

Parts (\ref{L_7810_SatCond_Tor}) and~(\ref{L_7810_SatCond_Van})
are similar to the argument of the previous paragraph,
but simpler.
Let $L \subseteq I + J$
be an ideal,
and observe as above that $L = (I \cap L) + (J \cap L)$.
The six term exact sequence in K-theory for the extension
\[
0 \longrightarrow I \cap J \cap L
\longrightarrow J \cap L
\longrightarrow (J \cap L) / (I \cap J \cap L)
\longrightarrow 0
\]
shows that $K_* \big( (J \cap L) / (I \cap J \cap L) \big)$ is torsion
(assuming~(\ref{L_7810_SatCond_Tor}))
or is trivial
(assuming~(\ref{L_7810_SatCond_Van})).
Then the six term exact sequence in K-theory for the extension
\[
0 \longrightarrow I \cap L
\longrightarrow (I \cap L) + (J \cap L)
\longrightarrow (J \cap L) / (I \cap J \cap L)
\longrightarrow 0
\]
shows that $K_* \big( (I \cap L) + (J \cap L) \big)$ is torsion
(assuming~(\ref{L_7810_SatCond_Tor}))
or is trivial
(assuming~(\ref{L_7810_SatCond_Van})).
\end{proof}

\begin{lem}\label{L_7810_StSatCond}
Each of the following properties satisfies
the conditions of Corollary~\ref{C_7809_ExtCtForMaxx}:
\begin{enumerate}
\item\label{L_7810_SatCond_Nuc}
Nuclearity.
\item\label{L_7810_StSatCond_PI}
Pure infiniteness for nonsimple \ca{s}.
\item\label{L_7810_StSatCond_StPI}
Strong pure infiniteness for nonsimple \ca{s},
provided that the zero \ca{} is included.
\item\label{L_7810_StSatCond_TDZ}
Topological dimension zero.
\item\label{L_7810_StSatCond_ResidHer}
Being residually hereditarily in
an upwards directed class of \ca{s}.
\end{enumerate}
\end{lem}

For separable \ca{s},
topological dimension zero is known to be of the form
residually hereditarily in~$\cC$
for an upwards directed class $\cC$ of \ca{s}.
See Theorem 2.10 of~\cite{PsnPh3}.
This is probably true in general,
but we don't need to know this to prove that
it satisfies
the conditions of Corollary~\ref{C_7809_ExtCtForMaxx}.

\begin{cor}\label{C_7828_GoodLg}
Let $(P)$ be any of the properties in Lemma~\ref{L_7810_StSatCond}.
Then every \ca~$A$ has a largest ideal $\LI{P}{A}$ with~$(P)$,
and $A / \LI{P}{A}$ has no nonzero ideals with~$(P)$.
\end{cor}

\begin{proof}
Combine Lemma~\ref{L_7810_StSatCond},
Corollary~\ref{C_7809_ExtCtForMaxx},
and Corollary~\ref{C_7810_LILI}.
\end{proof}

\begin{proof}[Proof of Lemma~\ref{L_7810_StSatCond}]
All the conditions are well known for nuclearity,
so we only consider the other four properties.

The zero \ca{} is purely infinite by convention;
see the discussion after Definition 4.1 of~\cite{KR}.
It is included in~(\ref{L_7810_StSatCond_StPI}) by definition.
It is trivial that the zero \ca{}
has both the other two properties.

The class of purely infinite \ca{s}
is closed under quotients and extensions by Theorem 4.19 of~\cite{KR},
and closed under arbitrary direct limits
by Proposition 4.18 of~\cite{KR}.
The class of strongly purely infinite \ca{s}
is closed under quotients by Proposition 5.11(i) of~\cite{KrRd2},
under extensions by Theorem 1.3 of~\cite{Kr3},
and under arbitrary direct limits
by Proposition 5.11(iv) of~\cite{KrRd2}.
Moreover,
the validity of none of these results
is affected by adding the zero \ca{} to the class.

For the rest of~(\ref{L_7810_StSatCond_TDZ}),
use Proposition~2.6 of~\cite{BP09}
to see that topological dimension zero
passes to extensions and quotients,
and use Lemma~2.6 of~\cite{PsnPh3}
and Lemma~3.6 of~\cite{PsnPh1}
to see that it passes to closures of increasing unions of ideals.

For the rest of~(\ref{L_7810_StSatCond_ResidHer}),
use instead Proposition 5.8 of~\cite{PsnPh2}
for extensions and quotients,
and use Proposition 5.9(2) of~\cite{PsnPh2}
for closures of increasing unions of ideals.
\end{proof}

\begin{proof}[Proof of Proposition~\ref{P_7809_IsMaxx}]
Stability is well known for type~I and nuclearity.
For stable rank one, it is Theorem 3.6 of~\cite{Rf},
and for real rank zero it is Corollary 3.3
and Corollary 2.8 of~\cite{BnPd0}.
For pure infiniteness it is Theorem 4.23 of~\cite{KR},
and for strong pure infiniteness it is
Proposition 5.11(iii) of~\cite{KrRd2}.
For conditions (\ref{P_7809_IsMaxx_Tor}),
(\ref{P_7809_IsMaxx_Van}),
(\ref{P_7809_IsMaxx_RRZTor}),
(\ref{P_7809_IsMaxx_RRZTriv}),
(\ref{P_7809_IsMaxx_Sr1Tor}),
and~(\ref{P_7809_IsMaxx_Sr1Triv}),
use what has been already observed,
the fact that stable isomorphism of algebras
gives stable isomorphism of their ideals,
and fact that stable isomorphism preserves K-theory.
Stable isomorphism preserves topological dimension zero
because it preserves the primitive ideal space.
Stable isomorphism preserves the condition
residually hereditarily in~$\cC$
by Proposition 5.12(2) of~\cite{PsnPh2}.

Passage to ideals is also well known for type~I and nuclearity.
For stable rank one, it is Theorem 4.4 of~\cite{Rf},
and for real rank zero it is Corollary 2.8 of~\cite{BnPd0}.
Given this, it is built into the definition
for (\ref{P_7809_IsMaxx_Tor}),
(\ref{P_7809_IsMaxx_Van}),
(\ref{P_7809_IsMaxx_RRZTor}),
(\ref{P_7809_IsMaxx_RRZTriv}),
(\ref{P_7809_IsMaxx_Sr1Tor}),
and~(\ref{P_7809_IsMaxx_Sr1Triv}).
Passage to ideals preserves pure infiniteness
by Theorem 4.19 of~\cite{KR}
and strong pure infiniteness by Proposition 5.11(ii) of~\cite{KrRd2}.
It preserves topological dimension zero
by Proposition~2.6 of~\cite{BP09},
and preserves the condition
residually hereditarily in~$\cC$
by Proposition~5.8 of~\cite{PsnPh2}.

Finally,
we prove that the properties admit largest ideals.
For type~I this is well known.
For nuclearity, pure infiniteness,
strong pure infiniteness,
topological dimension zero,
and residually hereditarily in
an upwards directed class,
use Lemma~\ref{L_7810_StSatCond}
and Corollary~\ref{C_7809_ExtCtForMaxx}.
For all the other properties,
use Lemma~\ref{L_7810_SatCond}
and Corollary~\ref{C_7809_StCtForMaxx}.
\end{proof}

It was already known that stable rank one admits largest ideals,
by Proposition~4.2 of~\cite{Rrd0},
since this result implies that the ideal ${\mathcal{K}}$
described there is the largest ideal with stable rank one.
It was also known that
real rank zero
admits largest ideals, by Theorem~2.3 of~\cite{BP2},
or by the
remark at the end of the proof of Proposition~2.5 of~\cite{BP2}.
In both cases,
these results give explicit descriptions of the largest ideal.
However, the line of reasoning used here
is needed to deal with the conditions
involving real and stable rank with K-theory restrictions.

\begin{lem}\label{L_7826_DSt}
Each of the following properties satisfies
the conditions of Corollary~\ref{C_7826_Sep}:
\begin{enumerate}
\item\label{L_7826_DSt_UCT}
All ideals satisfy the Universal Coefficient Theorem.
\item\label{L_7826_DSt_DSt}
$D$-stability for a \ssa{} \ca~$D$.
\end{enumerate}
\end{lem}

\begin{cor}\label{C_7828_SepGoodLg}
Let $(P)$ be any of the properties in Lemma~\ref{L_7826_DSt}.
Then every separable \ca~$A$ has a largest ideal $\LI{P}{A}$ with~$(P)$,
and $A / \LI{P}{A}$ has no nonzero ideals with~$(P)$.
\end{cor}

\begin{proof}
Combine Lemma~\ref{L_7826_DSt},
Corollary~\ref{C_7826_Sep},
and Corollary~\ref{C_7810_LILI}.
\end{proof}

\begin{proof}[Proof of Lemma~\ref{L_7826_DSt}]
We first consider the property
that all ideals satisfy the Universal Coefficient Theorem.
This is trivial for the zero \ca.
It is known that countable direct limits preserve
the Universal Coefficient Theorem,
and that ideals in direct limit algebras are direct limits of ideals
in the algebras in the system.
So condition~(\ref{C_7826_Sep_Dir})
in Corollary~\ref{C_7826_Sep} holds.

For the other two parts,
first recall that if we have a short exact sequence
of separable \ca{s}
in which two out of three of the algebras
satisfy the Universal Coefficient Theorem,
then so does the third.
Now let $A$ be a separable \ca{}
all of whose ideals satisfy the Universal Coefficient Theorem,
let $I \subseteq A$ be an ideal,
and let $L \subseteq A / I$ be an ideal.
Let $J$ be the inverse image of $L$ in~$A$.
Then $I$ and $J$ satisfy the Universal Coefficient Theorem,
so $L \cong J / I$ does too.
Next,
let $A$ be a separable \ca, let $I \subseteq A$ be an ideal,
suppose all ideals of $I$ and $A / I$
satisfy the Universal Coefficient Theorem,
and let $J \subseteq A$ be an ideal.
Let $L$ be the image of $J$ in~$A / I$.
Then $J \cap I$ is an ideal in~$I$
and $L$ is an ideal in $A / I$,
so both satisfy the Universal Coefficient Theorem.
Since $J / (J \cap I) \cong L$,
the ideal $J$ also satisfies the Universal Coefficient Theorem.

Now let $D$ be a \ssa{} \ca,
and consider $D$-stability.
The zero \ca{} is trivially $D$-stable.
The \ca~$D$ is automatically $K_1$-injective by Remark~3.3 of~\cite{Wn}.
We apply results of~\cite{TmWnt1};
note the blanket assumption on~$D$ at the
beginning of Section~3 of~\cite{TmWnt1},
which applies in both Section~3 and Section~4 there
and includes $K_1$-injectivity.
The class of separable $D$-absorbing \ca{s}
is closed under quotients by Corollary~3.3 of~\cite{TmWnt1},
under extensions by Theorem~4.3 of~\cite{TmWnt1},
and under direct limits indexed by $\N$
by Corollary~3.4 of~\cite{TmWnt1}.
\end{proof}

\begin{lem}\label{L_7812_UCT}
Each of the following properties is separably stable,
separably passes to ideals,
and separably admits largest ideals:
\begin{enumerate}
\item\label{L_7812_UCT_AF}
Being~AF.
\item\label{L_7812_UCT_ATZ}
Being an AT~algebra with real rank zero.
\item\label{L_7812_UCT_UCT}
All ideals satisfy the Universal Coefficient Theorem.
\item\label{L_7812_UCT_DSt}
$D$-stability for a \ssa{} \ca~$D$.
\end{enumerate}
\end{lem}

\begin{proof}
Stability is well known for the AF property.
For AT algebras with real rank zero,
it is clear that if $A$ is one then so is $K \otimes A$,
while the reverse implication follows
from Proposition~3 of~\cite{LnRrd}.
Separable stability for the Universal Coefficient Theorem
is well known,
and for the property we are considering
one adds the fact that stable isomorphism of algebras
implies stable isomorphism of their ideals.
Separable stability for the property of being $D$-absorbing
follows from Corollary~3.2 of~\cite{TmWnt1},
noting the blanket assumption at the
beginning of Section~3 there
and using Remark~3.3 of~\cite{Wn}
to see that the $K_1$-injectivity hypothesis
is automatic.

It is well known that ideals in AF~algebras are~AF.
Ideals in AT algebras with real rank zero are
AT algebras with real rank zero
by Proposition~3 of~\cite{LnRrd}.
For the property that all ideals in a separable \ca{}
satisfy the Universal Coefficient Theorem,
passage to ideals is built into the definition.
For $D$-absorbing,
use Corollary~3.3 of~\cite{TmWnt1},
with the same remarks as in the previous paragraph.

The properties of having all
ideals satisfy the Universal Coefficient Theorem
and being $D$-absorbing
separably admit largest ideals
by Lemma~\ref{L_7826_DSt}
and Corollary~\ref{C_7826_Sep}.

To prove that the other two properties separably admit largest ideals,
we use Lemma~\ref{L_7813_ExtCtForDSepMaxx}.
The zero algebra has all the properties.
We next consider the sum $I + J$
of ideals $I$ and $J$ in a \ca{} which have one of these properties.
We use the exact sequence
\begin{equation}\label{Eq_7813_SumSeq}
0 \longrightarrow I
\longrightarrow I + J
\longrightarrow J / (I \cap J)
\longrightarrow 0.
\end{equation}
In the AF case,
$J / (I \cap J)$ is a quotient of an AF~algebra
and hence AF.
So $I + J$ is AF because extensions of AF~algebras are AF
(\cite{Brn12}).
In the AT and real rank zero case,
we know that $I$ and $J$ have real rank zero
(by assumption)
and stable rank one
(true for all AT~algebras).
So $I + J$ has real rank zero by Proposition 2.14(iii) of~\cite{BP2}
and stable rank one by Proposition 2.14(ii) of~\cite{BP2}.
Also, $J / (I \cap J)$ is an AT~algebra with real rank zero,
because AT passes to quotients
(Proposition 2(ii) of~\cite{LnRrd})
and and real rank zero does too
(Theorem 3.14 of~\cite{BnPd0}).
So Theorem~5 of~\cite{LnRrd} implies that
$I + J$ is an AT~algebra with real rank zero.

It is well known that direct limits of sequences
of AF algebras are AF.
For AT algebras,
this is Proposition 2(i) of~\cite{LnRrd},
and for real rank zero it is Proposition 3.1 of~\cite{BnPd0}.
\end{proof}

\section{Largest ideals, fixed point algebras, and crossed
 products}\label{Sec_Consq}

\indent
The main results of this paper are Theorem~\ref{T_7809_FixedToCP}
and its corollary, Corollary~\ref{C_7812_MainCor},
which state that for properties of the types we are considering,
the fixed point algebra under a
second countable compact abelian group action
has the property \ifo{} the crossed product does.
Restricting to finite abelian groups,
this result is interesting primarily
when crossed products do not preserve the
property in question,
so nothing better can be expected.
In cases in which such crossed products do preserve the property,
or it is at least expected that they do,
we can say something else of interest.
See Theorem~\ref{T_7813_Equal},
where we show that in this case
$C^* (G, \, \LI{P}{A}, \, \af) = \LI{P}{C^* (G, A, \af)}$
and $\LI{P}{A} = \LI{P}{C^* (G, A, \af)} \cap A$.

\begin{lem}\label{L_7809_Invariance}
Let $(P)$ be a property of \ca{s}
that is isomorphism invariant and (separably) admits largest ideals.
Let $A$ be a (separable) \ca,
let $G$ be a group,
and let $\af \colon G \to \Aut (A)$
be an action of $G$ on~$A$.
Then $\LI{P}{A}$
is $\af$-invariant.
\end{lem}

\begin{proof}
Let $g \in G$.
Then $\af_g (\LI{P}{A})$ has~$(P)$,
so $\af_g (\LI{P}{A}) \subseteq \LI{P}{A}$
by definition.
\end{proof}

\begin{thm}\label{T_7809_FixedToCP}
Let $(P)$ be a property of \ca{s}
that separably passes to ideals,
is separably stable, and separably admits largest ideals.
Let $A$ be a separable \ca,
let $G$ be a second countable compact abelian group,
and let $\af \colon G \to \Aut (A)$
be an action of $G$ on~$A$.
Then $A^{\af}$ has~$(P)$ \ifo{} $C^* (G, A, \af)$ has~$(P)$.
\end{thm}

The proof is inspired by the proof of
Theorem~3.2 of~\cite{GtLz3}.

\begin{proof}[Proof of Theorem~\ref{T_7809_FixedToCP}]
We may assume that $A \neq 0$.

We identify $C (G, A)$ with multiplication given by
convolution with a dense subalgebra of $C^* (G, A, \af)$
in the usual way.
We assume that Haar measure on~$G$ has been normalized
to have total mass~$1$.
Define $\ph \colon A^{\af} \to C^* (G, A, \af)$
by sending $a \in A^{\af}$
to the constant function with value~$a$
in $C (G, A) \S C^* (G, A, \af)$.
Apply the result of~\cite{Rs}
to see that $\ph$
is an isomorphism from $A^{\af}$ to a corner $B$ of $C^* (G, A, \af)$.
Let $J \subseteq C^* (G, A, \af)$
be the ideal generated by~$B$.
Since $C^* (G, A, \af)$ is separable,
it follows from
Theorem 2.8 of~\cite{Bwn} that $J$ is stably isomorphic to $A^{\af}$.

Assume first that $C^* (G, A, \af)$ has~$(P)$.
Then $J$ has~$(P)$
since $(P)$ separably passes to ideals.
So $A^{\af}$ has~$(P)$,
since $A^{\af}$ is stably isomorphic to~$J$.

Now assume that $A^{\af}$ has~$(P)$.
Set $I = \LI{P}{C^* (G, A, \af)}$.
We want to show that $I = C^* (G, A, \af)$.
The ideal $J$ has~$(P)$
by stability,
so $J \subseteq I$.
Lemma~\ref{L_7809_Invariance} implies that
$I$ is invariant under the dual action~${\widehat{\af}}$.
So it is enough to show that
if $L \subseteq C^* (G, A, \af)$ is an ${\widehat{\af}}$-invariant ideal
which contains~$J$ then $L =  C^* (G, A, \af)$.

Theorem 3.4 of~\cite{GtLz2} provides
a $G$-invariant ideal $T \subseteq A$ such that $L = C^* (G, T, \af)$.
If $T \neq A$, then, using Lemma 3.1 of~\cite{GtLz3},
find $x \in A^{\af}$ such that $x \notin T$,
so that $\ph (x) \notin C^* (G, T, \af) = L$.
But $\ph (x) \in B \subseteq J \subseteq L$
by definition.
This contradiction shows that $T = A$, so $L = C^* (G, A, \af)$.
\end{proof}

We don't know whether Theorem~\ref{T_7809_FixedToCP}
holds for compact nonabelian groups.
To use the same method of proof,
one would need a version of Lemma~\ref{L_7809_Invariance}
for coactions.
For the property type~I,
this has been done in Section~2 of~\cite{GtLz3},
and yields Theorem~3.2 there.

\begin{cor}\label{C_7812_MainCor}
Let $(P)$ be any of the properties
in Proposition~\ref{P_7809_IsMaxx},
Remark~\ref{R_7809_ResidHer},
or Lemma~\ref{L_7812_UCT}.
Let $A$ be a separable \ca,
let $G$ be a second countable compact abelian group,
and let $\af \colon G \to \Aut (A)$
be an action of $G$ on~$A$.
Then $A^{\af}$ has~$(P)$ \ifo{} $C^* (G, A, \af)$ has~$(P)$.
\end{cor}

\begin{proof}
Combine Theorem~\ref{T_7809_FixedToCP}
with, as appropriate, Proposition~\ref{P_7809_IsMaxx},
Proposition \ref{P_7809_IsMaxx}(\ref{P_7809_IsMaxx_ResidHer})
and Remark~\ref{R_7809_ResidHer},
or Lemma~\ref{L_7812_UCT}.
\end{proof}

Corollary~\ref{C_7812_MainCor}
is not new for some of these properties,
at least for finite abelian groups~$G$.
Indeed,
for type~I, nuclearity,
the weak ideal property,
topological dimension zero
(when $A$ is separable),
and,
when $G$ is a finite abelian $2$-group,
for all of the properties in Remark~\ref{R_7809_ResidHer},
much more is already known.
Specifically,
for an action $\af \colon G \to \Aut (A)$,
\tfae:
\begin{enumerate}
\item\label{R_8712_Aaf}
$A^{\af}$ has~$(P)$.
\item\label{R_8712_A}
$A$ has~$(P)$.
\item\label{R_8712_CGAa}
$C^* (G, A, \af)$ has~$(P)$.
\end{enumerate}
For type~I and nuclearity,
this is well known.
For the weak ideal property,
(\ref{R_8712_Aaf}) implies~(\ref{R_8712_A})
by Theorem 8.9 in~\cite{PsnPh2},
(\ref{R_8712_A}) implies~(\ref{R_8712_CGAa})
by Corollary 8.10 of~\cite{PsnPh2},
and (\ref{R_8712_CGAa}) implies~(\ref{R_8712_Aaf})
because the weak ideal property passes to hereditary subalgebras
(use Proposition 5.10 of~\cite{PsnPh2})
and $A^{\af}$ is a hereditary subalgebra
of $C^* (G, A, \af)$
(by~\cite{Rs}).
For topological dimension zero,
assuming $A$ is separable,
the chain of implications is the same,
but now uses,
in the same order as above,
Theorem 3.6 of~\cite{PsnPh3},
Theorem 3.17 of~\cite{PsnPh1},
Lemma 3.3 of~\cite{PsnPh1}, and~\cite{Rs}.
When $G$ is a finite abelian $2$-group,
all the properties under consideration have the form
``residually hereditarily in an
upwards directed class of C*-algebras'',
and we use the same reasoning,
now applying, in order,
Theorem 3.2(ii) of~\cite{PsnPh3},
Corollary 3.3(ii) of~\cite{PsnPh3},
Proposition 5.10(2) of~\cite{PsnPh2}, and~\cite{Rs}.
We expect the restriction to $2$-groups to be unnecessary.

As often,
the ideal property does not fit.
Takai duality and Example~2.7 of~\cite{PsnPh1}
combine to
give an example in which (\ref{R_8712_A})
and~(\ref{R_8712_CGAa}) above hold,
but (\ref{R_8712_Aaf}) fails.

For many of the other properties
in Proposition~\ref{P_7809_IsMaxx},
the stronger result is definitely known to fail.
See the main result
(in Section~5) of~\cite{Bl0}
for AF,
Example~9 of~\cite{Ell3} for real rank zero,
Example 8.2.1 of~\cite{Bl0} for stable rank one,
and Theorem~4.8 of~\cite{Iz1} for all ideals have zero K-theory,
all ideals have torsion K-theory,
and the combination of real rank zero
and all ideals have zero or torsion $K_0$-groups.
The example for real rank zero
involves an AF algebra,
in particular, an AT algebra,
so shows that the stronger result fails for the condition of being
an AT~algebra with real rank zero.

The stronger result fails for stability under tensoring
with the strongly selfabsorbing \ca{}
$\bigotimes_{n = 1}^{\I} M_3$,
by Example 4.11 of~\cite{PhT4}.
It follows from the discussion there that the stronger result
in fact fails for stability under tensoring with any UHF algebra
of infinite type except possibly $\bigotimes_{n = 1}^{\I} M_2$.
Tensoring everything with ${\mathcal{O}}_{\I}$
(with the trivial action of the group where an action is needed),
one checks that the stronger result
fails for $D$-stability whenever $D$ is the tensor product
of ${\mathcal{O}}_{\I}$ and a UHF algebra
of infinite type except possibly $\bigotimes_{n = 1}^{\I} M_2$.
Presumably it also fails for
$\bigotimes_{n = 1}^{\I} M_2$
and ${\mathcal{O}}_{\I} \otimes \bigotimes_{n = 1}^{\I} M_2$.
The already cited Theorem~4.8 of~\cite{Iz1}
shows that the stronger result fails
for ${\mathcal{O}}_{2}$-stability.
We don't know about $Z$-stability or ${\mathcal{O}}_{\I}$-stability.

It is an apparently difficult open problem
whether pure infiniteness
and strong pure infiniteness
are preserved by crossed products by finite groups,
even for the special case $\Z / 2 \Z$.

For properties which are preserved by crossed products,
we can relate the largest ideal in the algebra
with the given property
to the largest ideal in the crossed product
with the same property.
We require that the group be discrete,
so that condition~(\ref{T_7811_IfCPPrsv_Inter})
in the next theorem
makes sense.
This is no great loss,
since very few properties are preserved by arbitrary
crossed products by a fixed compact group which is not discrete.
The proof works without finiteness of the group,
so we state the theorem in that generality,
but also very few properties are preserved by arbitrary
crossed products by a fixed infinite discrete group.
When the group is finite and the property is stable, we do better:
Theorem~\ref{T_7813_Equal}
shows that we get equality in parts (\ref{T_7811_IfCPPrsv_Contain})
and~(\ref{T_7811_IfCPPrsv_Inter}).

\begin{thm}\label{T_7811_IfCPPrsv}
Let $(P)$ be a property of \ca{s}
that is isomorphism invariant and admits largest ideals.
Let $G$ be a discrete abelian group.
Then \tfae:
\begin{enumerate}
\item\label{T_7811_IfCPPrsv_Prsv}
Whenever $A$ is a \ca{} with~$(P)$
and $\af \colon G \to \Aut (A)$
is any action of $G$ on~$A$,
then $C^* (G, A, \af)$ also has~$(P)$.
\item\label{T_7811_IfCPPrsv_Contain}
Whenever $A$ is a \ca{}
and $\af \colon G \to \Aut (A)$
is any action of $G$ on~$A$,
then $C^* (G, \, \LI{P}{A}, \, \af) \subseteq \LI{P}{C^* (G, A, \af)}$.
\item\label{T_7811_IfCPPrsv_Inter}
Whenever $A$ is a \ca{}
and $\af \colon G \to \Aut (A)$
is any action of $G$ on~$A$,
then $\LI{P}{A} \subseteq \LI{P}{C^* (G, A, \af)} \cap A$.
\setcounter{TmpEnumi}{\value{enumi}}
\end{enumerate}
\end{thm}

Condition~(\ref{T_7811_IfCPPrsv_Contain})
makes sense,
since Lemma~\ref{L_7809_Invariance} implies that
$\LI{P}{A}$ is $\af$-invariant.

\begin{proof}[Proof of Theorem~\ref{T_7811_IfCPPrsv}]
Assume~(\ref{T_7811_IfCPPrsv_Prsv});
we prove~(\ref{T_7811_IfCPPrsv_Contain}).
Let $A$ be a \ca,
and let $\af \colon G \to \Aut (A)$
be an action of $G$ on~$A$.
Since $\LI{P}{A}$ has~$(P)$,
the hypothesis implies that
the ideal $C^* (G, \, \LI{P}{A}, \, \af)$
of $C^* (G, A, \af)$ has~$(P)$.
So
$C^* (G, \, \LI{P}{A}, \, \af) \subseteq \LI{P}{C^* (G, A, \af)}$
by definition.

Now assume~(\ref{T_7811_IfCPPrsv_Contain});
we prove~(\ref{T_7811_IfCPPrsv_Inter}).
Let $A$ be a \ca,
and let $\af \colon G \to \Aut (A)$
be an action of $G$ on~$A$.
Using~(\ref{T_7811_IfCPPrsv_Contain}) at the second step,
we get
\[
\LI{P}{A}
 \subseteq C^* (G, \, \LI{P}{A}, \, \af) \cap A
 \subseteq \LI{P}{C^* (G, A, \af)} \cap A,
\]
as desired.

Finally, we prove that (\ref{T_7811_IfCPPrsv_Inter})
implies~(\ref{T_7811_IfCPPrsv_Prsv}).
So assume~(\ref{T_7811_IfCPPrsv_Inter}),
let $A$ be a \ca{} with~$(P)$,
and let $\af \colon G \to \Aut (A)$
be an action of $G$ on~$A$.
Since $A = \LI{P}{A}$,
from~(\ref{T_7811_IfCPPrsv_Inter}) we get
\begin{equation}\label{Eq_7811_AIn}
A \subseteq \LI{P}{C^* (G, A, \af)}.
\end{equation}
Also,
$\LI{P}{C^* (G, A, \af)}$ is invariant under the dual action
by Lemma~\ref{L_7809_Invariance},
so Theorem~3.4 of~\cite{GtLz2}
provides an $\af$-invariant ideal $J \subseteq A$
such that
\[
\LI{P}{C^* (G, A, \af)} = C^* (G, J, \af).
\]
Combining this with~(\ref{Eq_7811_AIn})
gives $J = A$.
So $\LI{P}{C^* (G, A, \af)} = C^* (G, A, \af)$,
that is, $C^* (G, A, \af)$ has~$(P)$.
\end{proof}

\begin{thm}\label{T_7813_Equal}
Let $(P)$ be a property of \ca{s}
that is stable and admits largest ideals.
Let $G$ be a finite abelian group.
Then the conditions of Theorem~\ref{T_7811_IfCPPrsv}
are also equivalent to the following:
\begin{enumerate}
\setcounter{enumi}{\value{TmpEnumi}}
\item\label{T_7811_IfCPPrsv_CpEq}
Whenever $A$ is a \ca{}
and $\af \colon G \to \Aut (A)$
is any action of $G$ on~$A$,
then $C^* (G, \, \LI{P}{A}, \, \af) = \LI{P}{C^* (G, A, \af)}$.
\item\label{T_7811_IfCPPrsv_InterEq}
Whenever $A$ is a \ca{}
and $\af \colon G \to \Aut (A)$
is any action of $G$ on~$A$,
then $\LI{P}{A} = \LI{P}{C^* (G, A, \af)} \cap A$.
\end{enumerate}
\end{thm}

\begin{proof}
It is enough to prove that the conditions
in Theorem~\ref{T_7811_IfCPPrsv}
together imply the conditions here.
So let $A$ be a \ca,
and let $\af \colon G \to \Aut (A)$
be an action of $G$ on~$A$.
Then $\LI{P}{C^* (G, A, \af)}$ is invariant under the dual action
by Lemma~\ref{L_7809_Invariance},
so Theorem~3.4 of~\cite{GtLz2}
provides an $\af$-invariant ideal $J \subseteq A$
such that
\begin{equation}\label{Eq_7813_Start}
\LI{P}{C^* (G, A, \af)} = C^* (G, J, \af).
\end{equation}
In this equation,
take crossed products by the dual action
${\widehat{\af}} \colon
 {\widehat{G}} \to \Aut \big( C^* (G, A, \af) \big)$.
Since ${\widehat{G}} \cong G$
and we are assuming
Theorem \ref{T_7811_IfCPPrsv}(\ref{T_7811_IfCPPrsv_Prsv}),
the left hand side of the result is an algebra with~$(P)$.
Set $n = \card (G)$;
then the right hand side is $M_n (J)$ by Takai duality.
So $M_n (J)$ has~$(P)$.
Therefore $J$ has~$(P)$ by stability.
So $J \subseteq \LI{P}{A}$.
Take crossed products by $G$,
use Theorem \ref{T_7811_IfCPPrsv}(\ref{T_7811_IfCPPrsv_Contain})
at the second step,
and use~(\ref{Eq_7813_Start}) at the third step,
getting
\[
C^* (G, J, \af)
 \subseteq C^* (G, \, \LI{P}{A}, \, \af)
 \subseteq \LI{P}{C^* (G, A, \af)}
 = C^* (G, J, \af).
\]
So $C^* (G, \, \LI{P}{A}, \, \af) = \LI{P}{C^* (G, A, \af)}$,
which is~(\ref{T_7811_IfCPPrsv_CpEq}).
It also follows that $J = \LI{P}{A}$.
Using this at the third step
and~(\ref{Eq_7813_Start}) at the first step,
we get
\[
\LI{P}{C^* (G, A, \af)} \cap A
 = C^* (G, J, \af) \cap A
 = J
 = \LI{P}{A}.
\]
This is~(\ref{T_7811_IfCPPrsv_InterEq}).
\end{proof}

\begin{cor}\label{C_7812_Contain}
\mbox{}
\begin{enumerate}
\item\label{C_7812_Contain_TDZ}
Let $(P)$ be the weak ideal property or topological dimension zero.
Let $G$ be a finite abelian group.
For every action $\af \colon G \to \Aut (A)$ of $G$
on a C*-algebra $A$, we have
\[
\LI{P}{A} = \LI{P}{C^* (G, A, \af)} \cap A.
\]
\item\label{C_7812_Contain_Up}
Let $(P)$ be
any of the properties in Remark~\ref{R_7809_ResidHer}.
Let $G$ be a finite abelian $2$-group.
For every action $\af \colon G \to \Aut (A)$ of $G$
on a C*-algebra $A$, we have
\[
\LI{P}{A} = \LI{P}{C^* (G, A, \af)} \cap A.
\]
\end{enumerate}
\end{cor}

\begin{proof}
We apply Theorem~\ref{T_7813_Equal}.
The properties involved are
stable and admit largest ideals
by parts (\ref{P_7809_IsMaxx_TDZ}) and~(\ref{P_7809_IsMaxx_ResidHer})
of Proposition \ref{P_7809_IsMaxx}
and Remark~\ref{R_7809_ResidHer}.
It remains only
to check that they are preserved by the appropriate crossed products.
In part~(\ref{C_7812_Contain_TDZ}),
for topological dimension zero
use Theorem 3.17 of~\cite{PsnPh1}
and for the weak ideal property use Corollary 8.10 of~\cite{PsnPh2}.
In part~(\ref{C_7812_Contain_Up}),
by Remark~\ref{R_7809_ResidHer}
all the properties
have the form
``residually hereditarily in an
upwards directed class of C*-algebras'',
so we can apply Corollary 3.3 of~\cite{PsnPh3}.
\end{proof}


\begin{thebibliography}{33}

\bibitem{Bl0} B.~Blackadar,
{\emph{Symmetries of the CAR algebra}},
Ann.\  Math.~(2) {\textbf{131}}(1990), 589--623.

\bibitem{Bwn} L.~G.\  Brown,
{\emph{Stable isomorphism of hereditary subalgebras of C*-algebras}},
Pacific J.~Math. {\textbf{71}}(1977), 335--348.

\bibitem{Brn12} L.~G.\  Brown,
{\emph{Extensions of AF~algebras: The projection lifting problem}},
pages 175--176 in:
{\emph{Operator Algebras and Applications}},
(R.~V.\  Kadison (ed.)),
Proceedings of Symposia in Pure Mathematics
{\textbf{38}}(1982), part~1.

\bibitem{BnPd0} L.~G.\  Brown and G.~K.\  Pedersen,
{\emph{C*-algebras of real rank zero}},
J.~Funct.\  Anal.\  {\textbf{99}}(1991), 131--149.

\bibitem{BnPd1} L.~G.\  Brown and G.~K.\  Pedersen,
{\emph{Extremal K-theory and index for C*-algebras}},
K-Theory  {\textbf{20}}(2000), 201--241.

\bibitem{BP2} L.~G.\  Brown and G.~K.\  Pedersen,
{\emph{Ideal structure and C*-algebras of low rank}},
Math.\  Scand.\  {\textbf{100}}(2007), 5--33.

\bibitem{BP09} L.~G.\  Brown and G.~K.\  Pedersen,
{\emph{Limits and C*-algebras of low rank or dimension}},
J.~Operator Theory {\textbf{61}}(2009), 381--417.

\bibitem{CP} J.~R. Carri\'{o}n and C.~Pasnicu,
{\emph{Approximations of C*-algebras and the ideal property}},
J.~Math.\  Anal.\  Appl.\  {\textbf{338}}(2008), 925--945.

\bibitem{Ell3} G.~A.\  Elliott, {\emph{A classification of
certain simple C*-algebras}}, pages 373--385 in: {\emph{Quantum and
Non-Commutative Analysis}},
H.~Araki etc.\  (eds.), Kluwer, Dordrecht, 1993.

\bibitem{GtLz2} E.~G.\  Gootman and A.~J.\  Lazar,
{\emph{Applications of noncommutative duality to
crossed product C*-algebras determined by an action or coaction}},
Proc.~London\  Math.\  Soc.\  {\textbf{59}}(1989), 593--624.

\bibitem{GtLz3} E.~G.\  Gootman and A.~J.\  Lazar,
{\emph{Compact group actions on C*-algebras: an application
of noncommutative duality}},
J.~Funct.\  Anal.\  {\textbf{91}}(1990), 237--245.

\bibitem{Iz1} M.~Izumi,
{\emph{Finite group actions on C*-algebras with the
   Rohlin property.~I}},
Duke Math.~J.\  {\textbf{122}}(2004), 233--280.

\bibitem{Kr3} E.~Kirchberg,
{\emph{On permanence properties of strongly purely infinite
  C*-algebras}},
preprint 2003
(Preprintreihe SFB 478--Geometrische Strukturen in der Mathematik,
Heft 284,
Westf\"{a}lische Wilhelmsuniversit\"{a}t M\"{u}nster; ISSN 1435-1188).

\bibitem{KR} E.~Kirchberg and M.~R{\o}rdam,
{\emph{Non-simple purely infinite C*-algebras}},
Amer.\  J.\  Math.\  {\textbf{122}}(2000), 637--666.

\bibitem{KrRd2} E.~Kirchberg and M.~R{\o}rdam,
{\emph{Infinite non-simple C*-algebras: absorbing the Cuntz algebra
${\mathcal{O}}_{\infty}$}},
Adv.\  Math.\  {\textbf{167}}(2002), 195--264.

\bibitem{LnRrd} H.~Lin and M.~R{\o}rdam,
{\emph{Extensions of inductive limits of circle algebras}},
J.~London Math.\  Soc.~(2) {\textbf{51}}(1995), 603--613.

\bibitem{PsnPh1}
C.~Pasnicu and N.~C.\  Phillips,
{\emph{Permanence properties
  for crossed products and fixed point algebras of finite groups}},
Trans.\  Amer.\  Math.\  Soc.\  {\textbf{366}}(2014), 4625--4648.

\bibitem{PsnPh2}
C.~Pasnicu and N.~C.\  Phillips,
{\emph{Crossed products by spectrally free actions}},
J.~Funct.\  Anal.\  {\textbf{269}}(2015), 915--967.

\bibitem{PsnPh3}
C.~Pasnicu and N.~C.\  Phillips,
{\emph{The weak ideal property and topological dimension zero}},
Canadian J.\  Math., to appear.

\bibitem{PhT4} N.~C.\  Phillips,
{\emph{Finite cyclic group actions with the
tracial Rokhlin property}},
Trans.\  Amer.\  Math.\  Soc.\  {\textbf{367}}(2015), 5271--5300.

\bibitem{Rf} M.~A.\  Rieffel,
{\emph{Dimension and stable rank in the K-theory of C*-algebras}},
Proc.\  London Math.\  Soc.~(3) {\textbf{46}}(1983), 301--333.

\bibitem{Rrd0} M.~R{\o}rdam,
{\emph{Advances in the theory of unitary rank and regular
 approximation}},
Ann.\  Math.\  (2) {\textbf{128}}(1988), 153--172.

\bibitem{Rrd} M.~R{\o}rdam,
{\emph{Classification of nuclear, simple C*-algebras}},
pages 1--145 of:
M.~R{\o}rdam and E.~St{\o}rmer,
{\emph{Classification of nuclear C*-algebras. Entropy in operator
algebras}},
Encyclopaedia of Mathematical Sciences vol.\  126,
Springer-Verlag, Berlin, 2002.

\bibitem{Rs} J.~Rosenberg,
{\emph{Appendix to O.~Bratteli's paper on
``Crossed products of UHF algebras''}},
Duke Math.~J.\  {\textbf{46}}(1979), 25--26.

\bibitem{TmWnt1} A.~S.\  Toms and W.~Winter,
{\emph{Strongly self-absorbing C*-algebras}},
Trans.\  Amer.\  Math.\  Soc.\  {\textbf{359}}(2007), 3999--4029.

\bibitem{Wn} W.~Winter,
{\emph{Strongly self-absorbing C*-algebras are ${\mathcal{Z}}$-stable}},
J.~Noncommut.\  Geom.\   {\textbf{5}}(2011), 253--264.

\end{thebibliography}
\end{document}